\documentclass[a4paper,11pt,reqno]{amsart}
\usepackage{graphicx, psfrag, amsmath,  amssymb, array, hyperref}
\usepackage{xcolor,colortbl} 

\newtheorem{lem}{Lemma}[section]
\newtheorem{thm}[lem]{Theorem}
\newtheorem{pro}[lem]{Proposition}
\newtheorem{cor}[lem]{Corollary}
\newtheorem{exa}[lem]{Example}

\newtheorem{defi}[lem]{Definition}

\numberwithin{equation}{section}

\newcommand{\A}{{\mathcal{A}}}

\newcommand{\C}{{\mathcal{C}}}

\newcommand{\des}{{\textsf{des}}}

\newcommand{\art}{{\mathsf{art}}}

\newcommand{\drop}{{\mathsf{drop}}}

\newcommand{\Hop}{{\mathrm{Hop}}}

\newcommand{\PP}{{\mathcal{P}}}
\newcommand{\U}{\mathsf{U}}
\renewcommand{\L}{\mathsf{L}}
\newcommand{\DD}{{\mathsf{D}}}

\newcommand{\W}{{\mathcal{W}}}
\newcommand{\X}{{\mathcal{X}}}

\newcommand{\M}{{\mathcal{M}}}
\newcommand{\R}{{\mathcal{R}}}
\newcommand{\SSS}{{\mathcal{S}}}

\title[$\gamma$-positivity for a Refinement of Median Genocchi Numbers]{Gamma-positivity for a Refinement of\\ Median Genocchi Numbers}

\author[{S.-P. Eu}]{Sen-Peng Eu}
\address{Department of Mathematics, National Taiwan Normal University, Taipei 116325, and Chinese Air Force Academy, Kaohsiung 820009, Taiwan, ROC}
\email{speu@math.ntnu.edu.tw}

\author[T.-S. Fu]{Tung-Shan Fu}
\address{Department of Applied Mathematics, National Pingtung University, Pingtung 900391, Taiwan, ROC}
\email{tsfu@mail.nptu.edu.tw}

\author[H.-H. Lai]{Hsin-Hao Lai}
\address{Department of Mathematics, National Kaohsiung Normal University, Kaohsiung 824004, Taiwan, ROC}
\email{hsinhaolai@nknu.edu.tw}

\author[Y.-H. Lo]{Yuan-Hsun Lo}
\address{Department of Applied Mathematics, National Pingtung University, Pingtung 900391, Taiwan, ROC}
\email{yhlo@mail.nptu.edu.tw}

\begin{document}

\begin{abstract}
We study the generating function of descent numbers for the permutations with descent pairs of prescribed parities, the distribution of which turns out to be a refinement of median Genocchi numbers. We prove the $\gamma$-positivity for the polynomial and derive the generating function for the $\gamma$-vectors, expressed in the form of continued fraction. We also come up with an artificial statistic that gives a $q$-analogue of the $\gamma$-positivity for the permutations with descents only allowed from an odd value to an odd value.
\end{abstract}

\maketitle

\section{Introduction} 

\subsection{Genocchi numbers and median Genocchi numbers}
The (signless) \emph{Genocchi numbers} $\{g_n\}_{n\ge 1}=\{1,1,3,17,155,2073,\dots\}$ \cite[A110501]{oeis}, which are in relation to Bernoulli numbers $B_{2n}$, namely $g_n=2(1-2^{2n})(-1)^nB_{2n}$, can be defined by their exponential generating function \cite[page 305]{Du-1974}
\[
\sum_{n\ge 1} g_n\frac{x^{2n}}{(2n)!}=x\,\tan\frac{x}{2}.
\]
Let $\mathfrak{S}_n$ be the set of permutations of $[n]:=\{1,2,\dots, n\}$. Among numerous combinatorial interpretations of Genocchi numbers, $g_{n+1}$ counts the following four kinds of  \emph{Dumont permutations}:
\begin{enumerate}
\item the number of $\sigma\in\mathfrak{S}_{2n}$ such that if $\sigma(i)$ is even then $\sigma(i)>\sigma(i+1)$ and $i<2n$, otherwise $\sigma(i)<\sigma(i+1)$ or $i=2n$,
\item the number of $\sigma\in\mathfrak{S}_{2n}$ such that $2i>\sigma(2i)$ and $2i-1\le\sigma(2i-1)$ for all $i\in [n]$,
\item the number of $\sigma\in\mathfrak{S}_{2n}$ such that if $\sigma(i)>\sigma(i+1)$ then both of $\sigma(i)$ and $\sigma(i+1)$ are even for all $i\in [2n-1]$, and
\item the number of $\sigma\in\mathfrak{S}_{2n}$ such that if $i>\sigma(i)$ then both of $i$ and $\sigma(i)$ are even for all $i\in [2n]$.
\end{enumerate}
The objects (i), (ii) are due to Dumont \cite{Du-1974}, and  (iii), (iv) are given by Burstein, Josuat-Verg\`es, and Stromquist \cite{BJS}.

The (signless) \emph{median Genocchi number} $\{h_n\}_{n\ge 0}=\{1, 2$, $8, 56, 608$, $9440, \dots\}$ \cite[A005439]{oeis} can be defined combinatorially in terms of \emph{Dumont derangements} \cite[Corollary 2.4]{DR}, i.e., $h_n$ is the number of $\sigma\in\mathfrak{S}_{2n+2}$ such that $2i>\sigma(2i)$ and $2i-1<\sigma(2i-1)$ for all $i\in [n+1]$. According to Lazar and Wachs \cite[Corollary 6.2]{LW}, $h_n$ also counts the number of $\sigma\in\mathfrak{S}_{2n}$ such that if $i>\sigma(i)$ then $i$ is odd and $\sigma(i)$ is even for all $i\in [2n]$. Recently, Hetyei proved that the number of regions in the homogenized Linial arrangement is counted by median Genocchi number \cite{Hetyei}.

\subsection{Gamma-positivity for palindromic polynomials}
A polynomial $A(t)=a_0+a_1t+\cdots+a_nt^n$ with palindromic coefficients (i.e., $a_{n-i}=a_i$) can be written as a sum of the polynomials of the form $t^j(1+t)^{n-2j}$, 
\[
A(t)=\sum_{j=0}^{\lfloor n/2\rfloor} \gamma_{j}\, t^j(1+t)^{n-2j}.
\]
The coefficients $\gamma_j$ form a sequence called the $\gamma$-\emph{vector}. The palindromic polynomial $A(t)$ is said to be $\gamma$-\emph{positive} if $\gamma_j\ge 0$ for all $j$. 

For any $\sigma\in\mathfrak{S}_n$, let $\sigma=\sigma_1\sigma_2\cdots\sigma_n$, where $\sigma_i=\sigma(i)$ for $1\le i\le n$.  A \emph{descent} in $\sigma$ is an $i$ such that  $\sigma_i>\sigma_{i+1}$, $1\le i\le n-1$.  Here the element $\sigma_i$ ($\sigma_{i+1}$, respectively) is called a \emph{descent top} (\emph{descent bottom}, respectively), and the ordered pair $(\sigma_i,\sigma_{i+1})$ is called a \emph{descent pair}. Moreover, the descent pair $(\sigma_i,\sigma_{i+1})$ is called \emph{even-odd} (\emph{odd-even}, \emph{odd-odd}, and \emph{even-even}, respectively) if the parities of $(\sigma_i, \sigma_{i+1})$  is (even, odd) ((odd, even), (odd, odd), and (even, even), respectively). Let $\des(\sigma)$ denote the number of descents of $\sigma$.

It was first proved by Foata and Sch\"utzenberger \cite{FS} that the $n$th \emph{Eulerian polynomial} is $\gamma$-positive, i.e., 
\begin{equation} \label{eqn:Eulerian_polynomial}
\sum_{\sigma\in\mathfrak{S}_n} t^{\des(\sigma)}=\sum_{j=0}^{\lfloor (n-1)/2\rfloor} \gamma_{n,j}\, t^j(1+t)^{n-1-2j},
\end{equation}
where $\gamma_{n,j}$ is the number of $\sigma\in\mathfrak{S}_n$ with $\des(\sigma)=j$,  $\sigma_1<\sigma_2$, and no double descents (defined in subsection 1.4). Foata and Strehl \cite{Foata-Strehl} gave an interesting combinatorial proof of this result; see also \cite[Chapter 4]{Petersen}. Various $q$-analogues of Eq.\,(\ref{eqn:Eulerian_polynomial}) appeared in \cite{HJZ,LZ,LSW}.
For the $\gamma$-positivity of generalized Eulerian polynomials, see some results in \cite{Branden_04,Stembridge_97} for posets, and Gal's result \cite{Gal_05} for combinatorial invariants of flag simplicial sphere. 

In this paper we study the $\gamma$-positivity for the generating function of descent numbers for the permutations with descent pairs of prescribed parities, the distribution of which turns out to be a refinement of median Genocchi numbers.

\subsection{On permutations with only even-odd descent pairs}
Let $\X_{n}$ be the set of permutations in $\mathfrak{S}_{n}$ that contain only even-odd descent pairs. Notice that $|\X_{2n}|=h_n$, the $n$th median Genocchi number; see Proposition \ref{pro:descent-to-drop-bijection-even-odd}. Define
\[
X_n(t):=\sum_{\sigma\in\X_{2n}} t^{\des(\sigma)},
\]
the \emph{descent polynomial} for $\X_{2n}$. Several of the initial polynomials are listed below:
\begin{align*}
X_1(t) &= 1+t,\\
X_2(t) &= 1+6t+t^2, \\
X_3(t) &= 1+27t+27t^2+t^3, \\
X_4(t) &= 1+112t++382t^2+112t^3+t^4, \\
X_5(t) &= 1+453t+4266t^2+4266t^3+453t^4+t^5.
\end{align*}
Some different refinements of median Genocchi numbers have been studied \cite{Feigin,HZ-DM,ZZ}. Notice that these palindromic polynomials can be written as follows:
\begin{align*}
X_1(t) &= 1+t,\\
X_2(t) &= (1+t)^2+4t, \\
X_3(t) &= (1+t)^3+24t(1+t), \\
X_4(t) &= (1+t)^4+108t(1+t)^2+160t^2, \\
X_5(t) &= (1+t)^5+448t(1+t)^3+2912t^2(1+t).
\end{align*}
One of our main results is the following $\gamma$-positivity for the descent polynomial for $\X_{2n}$ (Theorem \ref{thm:even-odd-descent-result}). We remark that the interpretation of the $\gamma$-vector (Definition \ref{def:primary-even-odd-descent}) is quite different from the permutations with `no-double-descent' feature for the $\gamma$-vector in Eq.\,(\ref{eqn:Eulerian_polynomial}).
Here we use the notation $\frac{\alpha_1}{\beta_1}{{}\atop{-}}\frac{\alpha_2}{\beta_2}{{}\atop{-}}\cdots=\alpha_1/(\beta_1-\alpha_2/(\beta_2-\cdots))$ for continued fractions.

\begin{defi} \label{def:primary-even-odd-descent} {\rm For any $\sigma\in\X_{2n}$ with descent tops $\{t_1,\dots,t_k\}$ and descent bottoms $\{b_1,\dots,b_k\}$  ($k\ge 0$), we say that $\sigma$ is a \emph{primary even-odd-descent permutation} if for any $i,j$,
\begin{equation} \label{eqn:descent_top-bottom}
t_i>b_j\Rightarrow  t_i-b_j\ge 3. 
\end{equation}
}
\end{defi}

\smallskip
\begin{thm} \label{thm:even-odd-descent-result} For all $n\ge 1$, the descent polynomial for $\X_{2n}$ can be expanded as
\begin{equation} \label{eqn:X_n(t)}
X_n(t)=\sum_{j=0}^{\lfloor n/2\rfloor} \gamma_{n,j}\, t^j(1+t)^{n-2j},
\end{equation}
where $\gamma_{n,j}$ is the number of primary even-odd-descent permutations in $\X_{2n}$ with $j$ descents. Moreover, the generating function for $\gamma_{n,j}$ can be expressed in the form of continued fraction as
\begin{equation} \label{eqn:CF-expansion}
\sum_{n\ge 0}\left(\sum_{j=0}^{\lfloor n/2\rfloor} \gamma_{n,j} t^j \right)x^n=\frac{1}{1-\mu_0 x}{{}\atop{-}}\frac{\lambda_1 x^2}{1-\mu_1 x}{{}\atop{-}}\frac{\lambda_2 x^2}{1-\mu_2 x}{{}\atop{-}}\frac{\lambda_3 x^2}{1-\mu_3x}{{}\atop{-}}\cdots,
\end{equation}
where
\[
\left\{\begin{array}{lll}
\mu_h &= (h+1)^2, &\mbox{for $h\ge 0$}; \\ [0.8ex]\lambda_{h} &= h^2(h+1)^2 t, &\mbox{for $h\ge 1$.}
       \end{array}
\right.
\]
\end{thm}

\smallskip
The initial terms of the expansion of Eq.\,(\ref{eqn:CF-expansion}) are 
\begin{align*}
&\frac{1}{1-x}{{}\atop{-}}\frac{4t x^2}{1-4x}{{}\atop{-}}\frac{36t x^2}{1-9 x}{{}\atop{-}}\frac{144t x^2}{1-16x}{{}\atop{-}}\cdots \\
&\quad =1+x+(1+4t)x^2+(1+24t)x^3+(1+108t+160t^2)x^4+(1+448t+2912t^2)x^5 \\
&\quad\qquad +(1+1812t+35520t^2+27136t^3)x^6+(1+7272t+370496t^2+1106944t^3)x^7+\cdots.
\end{align*}
Most recently,  Pan and Zeng \cite[Theorem 13]{PZ} obtained a $(p,q)$-analogue of the descent polynomial $X_n(t)$ in the spirit of Br\"and\'en's $(p,q)$-Eulerian polynomials \cite{Branden_08}.

\subsection{On permutations with only odd-odd descent pairs}
Let $\W_{n}$ be the set of permutations in $\mathfrak{S}_n$ that contain only odd-odd descent pairs. Notice that $|\W_{2n+1}|=g_{n+2}$, which is in connection with the Dumont permutations in $\mathfrak{S}_{2n+2}$ that contain only even-even descent pairs \cite{BJS}. Let $\W^*_{2n+1}$ be the subset of $\W_{2n+1}$ consisting of the permutations with an odd last element. Notice that $|\W^*_{2n+1}|=|\X_{2n}|$ and $\W^*_{2n+1}$ shares the same descent polynomial with $\X_{2n}$; see Theorem \ref{thm:connection-X(2n)-W*(2n+1)}. 
We come up with an artificial statistic that gives a $q$-analogue of $\gamma$-positivity for the descent polynomial for $\W^*_{2n+1}$.

For any $\sigma=\sigma_1\cdots\sigma_n\in \mathfrak{S}_n$, the element $\sigma_i$ ($1\le i\le n$) is a \emph{peak} (\emph{valley}, \emph{double descent}, \emph{double ascent}, respectively) if $\sigma_{i-1}<\sigma_i>\sigma_{i+1}$ ($\sigma_{i-1}>\sigma_i<\sigma_{i+1}$, $\sigma_{i-1}>\sigma_i>\sigma_{i+1}$, $\sigma_{i-1}<\sigma_i<\sigma_{i+1}$, respectively), where we use the convention $\sigma_0=\sigma_{n+1}=0$. In particular, the element $\sigma_1$ ($\sigma_n$, respectively) is a peak if $\sigma_1>\sigma_2$ ($\sigma_{n-1}<\sigma_{n}$, respectively).

\begin{defi} \label{def:art} {\rm
Let $\sigma\in\mathfrak{S}_n$. For each $i\in [n]$, let $v(i)$ ($p(i)$, respectively) be the number of valleys (peaks, respectively) less than $i$ and on the left of the element $i$ in $\sigma$. Define
\begin{equation} \label{eqn:art}
\art(\sigma):=\sum_{i=1}^{n} \big( v(i)-p(i) \big).
\end{equation}
}
\end{defi}
For example, given $\sigma=5\: 1\: 2\: 4\: 6\: 7\: 3\: 8\: 9\in\mathfrak{S}_9$ with peaks $\{5,7,9\}$ and valleys $\{1,3\}$, the values of $v(i)$ and $p(i)$ are shown below. We have $\art(\sigma)=3.$
\[
\begin{array}{c|ccccccccc}

 i & 1 &  2 & 3 & 4 & 5 & 6 & 7 &  8  & 9\\
\hline
 v(i)  & 0 & 1 & 1 & 1 & 0 & 1 & 1 & 2  & 2\\
 p(i)  & 0 & 0 & 0 & 0 & 0 & 1 & 1 & 2  & 2\\

\end{array}
\]
For $n\ge 1$, define
\[
W_n(q,t):=\sum_{\sigma\in\W^*_{2n+1}} q^{\art(\sigma)} t^{\des(\sigma)},
\]
a $q$-generalization of the descent polynomial  for $\W^*_{2n+1}$. For example,
\begin{align*}
W_1(q,t) &= 1+t,\\
W_2(q,t) &= 1+(3+2q+q^2)t+t^2, \\
W_3(q,t) &= 1+(6+8q+8q^2+4q^3+q^4)t+(6+8q+8q^2+4q^3+q^4)t^2+t^3.
\end{align*}
The polynomials for $n=2$ and $n=3$ can be written as follows:
\begin{align*}
W_2(q,t) &= (1+t)^2+(1+2q+q^2)t, \\
W_3(q,t) &= (1+t)^3+(3+8q+8q^2+4q^3+q^4)t(1+t).
\end{align*}

\smallskip
Our second main result is the following $q$-$\gamma$-positivity for the descent polynomial  for $\W^*_{2n+1}$.
We use the notation $[k]_q:=1+q+\cdots+q^{k-1}$ for all positive integers $k$.

\smallskip
\begin{defi} \label{def:primary-odd-odd-descent} {\rm 
For any $\sigma=\sigma_1\cdots\sigma_{2n+1}\in\W^*_{2n+1}$, we say that $\sigma$ is a \emph{primary odd-odd-descent permutation} if $\sigma$ contains no double descent, and the last entry of $\sigma$ is a peak, i.e., $\sigma_{2n}<\sigma_{2n+1}$.
}
\end{defi}

\smallskip
\begin{thm} \label{thm:odd-odd-descent-result} For all $n\ge 1$, the descent $q$-polynomial for $\W^*_{2n+1}$ can be expanded as
\begin{equation} \label{eqn:W_n(t)}
W_n(q,t)=\sum_{j=0}^{\lfloor n/2\rfloor} \gamma_{n,j}(q)\: t^j(1+t)^{n-2j},
\end{equation}
where
\[
\gamma_{n,j}(q)=\sum_{\sigma} q^{\art(\sigma)},
\]
and the sums run through all primary odd-odd-descent permutations in $\W^*_{2n+1}$ with $j$ descents.
Moreover, the generating function for $\gamma_{n,j}(q)$ can be expressed as
\begin{equation} \label{eqn:q-CF-expansion}
\sum_{n\ge 0}\left(\sum_{j=0}^{\lfloor n/2\rfloor} \gamma_{n,j}(q)\: t^j \right)x^n=\frac{1}{1-\mu_0 x}{{}\atop{-}}\frac{\lambda_1 x^2}{1-\mu_1 x}{{}\atop{-}}\frac{\lambda_2 x^2}{1-\mu_2 x}{{}\atop{-}}\frac{\lambda_3 x^2}{1-\mu_3x}{{}\atop{-}}\cdots,
\end{equation}
where
\[
\left\{\begin{array}{lll}
\mu_h &= [h+1]_q^2, &\mbox{for $h\ge 0$}; \\ [0.8ex]\lambda_{h} &= [h]_q^2[h+1]_q^2 t, &\mbox{for $h\ge 1$.}
       \end{array}
\right.
\]
\end{thm}

\smallskip
The initial terms of the expansion of Eq.\,(\ref{eqn:q-CF-expansion}) are 
\begin{align*}
&\frac{1}{1-[1]_q^2x}{{}\atop{-}}\frac{[1]_q^2[2]_q^2t x^2}{1-[2]_q^2x}{{}\atop{-}}\frac{[2]_q^2[3]_q^2t x^2}{1-[3]_q^2 x}{{}\atop{-}}\frac{[3]_q^2[4]_q^2t x^2}{1-[4]_q^2x}{{}\atop{-}}\cdots \\
&\quad =1+x+\big(1+(1+2q+q^2)t\big)x^2+\big(1+(3+8q+8q^2+4q^3+q^4)t\big)x^3+\cdots.
\end{align*}
Pan and Zeng \cite{PZ} obtained a multivariate generalization of the $q$-$\gamma$-positivity for $\W^*_{2n+1}$.

\smallskip
The rest of the paper is organized as follows. In section 2 we prove Theorem \ref{thm:even-odd-descent-result} and a byproduct of $\gamma$-positivity for the permutations with only even-odd drops (Corollary \ref{cor:even-odd-drop-corollary}).  In section 3 we prove Theorem \ref{thm:odd-odd-descent-result} in a similar manner and present an algorithmic bijection between the objects in Theorem \ref{thm:even-odd-descent-result} and Theorem \ref{thm:odd-odd-descent-result}. For the permutations with descent pairs of the remaining parities, we mention analogous results in section 4.

\section{On permutations with only even-odd descents}
In this section we shall give a combinatorial proof of Theorem \ref{thm:even-odd-descent-result}, which involves a ``hopping'' operation used in \cite{Branden_08,Foata-Strehl}.

\subsection{Inter-hopping operation}
Let $\sigma\in \X_{2n}$. For each $i\in [2n]$, the element $i$ is called \emph{saturated} in $\sigma$ if it is a descent top or a descent bottom of $\sigma$. 
For each $j\in [n]$, the pair $\{2j-1,2j\}$ of elements is called \emph{free} in $\sigma$ if both of $2j-1$ and $2j$ are saturated, or neither of $2j-1$ and $2j$ is saturated. 
By Definition \ref{def:primary-even-odd-descent}, notice that $\sigma$ is a primary even-odd descent permutation if for each $j\in [n]$ the two elements $2j-1$ and $2j$ are not simultaneously saturated. 
We partition the permutations in $\X_{2n}$ into equivalence classes by an \emph{inter-hopping} operation on free pairs of elements.

Given $\omega=x_1\cdots x_{2n}\in\X_{2n}$ with a free pair $\{2r-1,2r\}$, let  $\{x_a,x_b\}=\{2r-1,2r\}$ for some $a<b$.  We shall construct a permutation $\omega'\in\X_{2n}$ with $|\des(\omega')-\des(\omega)|=1$ by the following process.

\smallskip
\noindent
{\bf Algorithm A.}

\begin{enumerate}
\item[(A1)] Neither of $2r-1$ and $2r$ is saturated. If the elements $2r-1$ and $2r$ are adjacent in $\omega$ then $\omega'$ is obtained from $\omega$ by switching $2r-1$ and $2r$. Otherwise, we factorize $\omega$  as
\[
\omega=\cdots \beta_0\, x_a\, \alpha_1\beta_1\alpha_2\beta_2\cdots\alpha_d\beta_d\, x_b\, \alpha_{d+1}\cdots,
\]
where $\alpha_j$ ($\beta_j$, respectively) is a maximal sequence of consecutive entries greater than $2r$ (less than $2r-1$, respectively). (The sequences $\beta_0$ and $\alpha_{d+1}$ are possibly empty.) There are two cases for $\omega'$:

\begin{enumerate}
\item $x_a=2r$ and $x_b=2r-1$. Then set
\[
\omega':=\cdots\beta_0\alpha_1\,(2r-1)\,\alpha_2\beta_1\alpha_3\beta_2\cdots\alpha_{d}\beta_{d-1}\,(2r)\,\beta_{d}\alpha_{d+1}\cdots.
\]
\item $x_a=2r-1$ and $x_b=2r$. Then set
\[
\omega':=\cdots\beta_0\, (2r)\,\beta_1\alpha_1\beta_2\alpha_2\cdots\beta_{d}\alpha_{d}\,(2r-1)\,\alpha_{d+1}\cdots.
\]
\end{enumerate}
\item[(A2)] Both of $2r-1$ and $2r$ are saturated. The construction of $\omega'$ is exactly the reverse operation of (A1).
\end{enumerate}

\begin{exa} \label{exa:hopping-example} {\rm On the left of Figure \ref{fig:hopping} is the permutation $\omega=2\: 5\: 10\: 14\: 3\: 4\: 6\: 1\: 8\: 11\: 12\: 7\: 9\: 13\in\X_{14}$, with a free pair $\{9,10\}$. We factorize $\omega$ as $\omega=\beta_0\:(10)\:\alpha_1\beta_1\alpha_2\beta_2\:(9)\:\alpha_3$, where $\beta_0=2\: 5$, $\alpha_1=14$, $\beta_1=3\: 4\: 6\: 1\: 8$, $\alpha_2=11\: 12$, $\beta_2=7$, and $\alpha_3=13$. By (A1)(i), $\omega'=2\: 5\: 14\: 9\: 11\: 12\: 3\: 4\: 6\: 1\: 8\: 10\: 7\: 13$, as shown on the right of Figure \ref{fig:hopping}. Moreover, if $\omega=2\: 5\: 9\: 14\: 3\: 4\: 6\: 1\: 8\: 11\: 12\: 7\: 10\: 13$ then by (A1)(ii), $\omega'=2\: 5\: 10\: 3\: 4\: 6\: 1\: 8\: 14\: 7\: 11\: 12\: 9\: 13$, as shown in Figure \ref{fig:hopping-2}. Notice that $\des(\omega')=\des(\omega)+1$.
}
\end{exa}

\begin{figure}[ht]
\begin{center}
\includegraphics[width=5.4in]{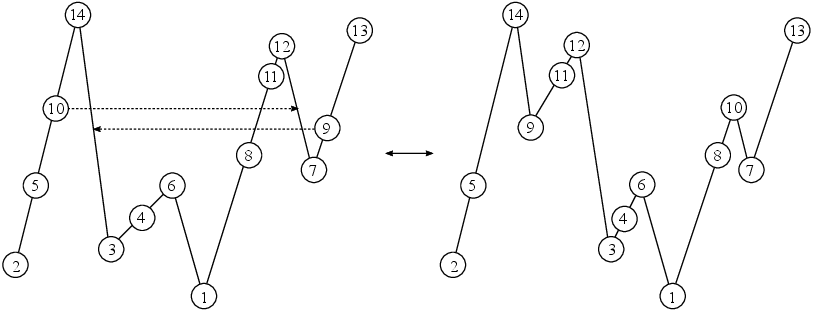}
\end{center}
\caption{\small An illustration for Example \ref{exa:hopping-example}.}
\label{fig:hopping}
\end{figure}

\begin{figure}[ht]
\begin{center}
\includegraphics[width=5.4in]{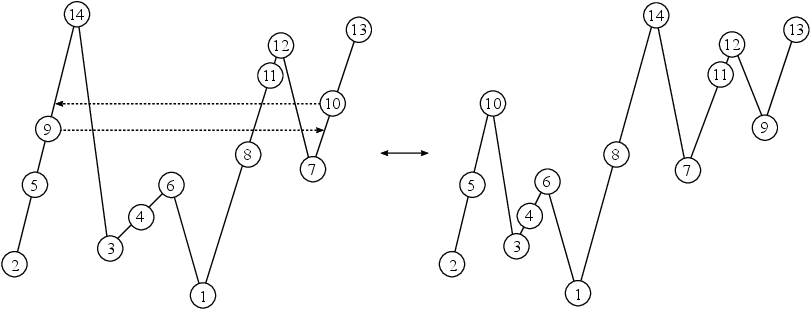}
\end{center}
\caption{\small An illustration for Example \ref{exa:hopping-example}.}
\label{fig:hopping-2}
\end{figure}

Let $\Hop(\omega)$ be the hop-equivalence class of $\omega$. 
Setting each free pair of $\omega$ unsaturated yields the unique primary even-odd-descent permutation, say $\pi$, in $\Hop(\omega)$. 
That is, for each $r\in [n]$ at most one of the two elements $2r-1$ and $2r$ in $\pi$ is saturated. Note that $\pi$ is the permutation in $\Hop(\omega)$ having the minimum number of descents.
Moreover, if $\des(\pi)=j$ then there are $n-2j$ free pairs in $\pi$. Hence we have
\begin{equation} \label{eqn:representative}
\sum_{\sigma\in \Hop(\omega)} t^{\des(\sigma)}=t^{\des(\pi)}(1+t)^{n-2\,\des(\pi)}.
\end{equation}

\subsection{Encoding primary even-odd-descent permutations} We enumerate the primary even-odd-descent permutations by a classification of their descent tops and descent bottoms (Proposition \ref{pro:product-s(i)}). 

\begin{defi} \label{def:signature} {\rm
A set $S$ is called a \emph{signature} in $[2n]$ if it satisfies the following conditions.
\begin{enumerate}
\item The set $S$ consists of $k$ odd elements and $k$ even elements for some $k\ge 0$.
\item For each $j\in [n]$, at most one of the two elements $2j-1$ and $2j$ is in $S$.
\item  For each $i\in [2n]$, the number of odd elements is greater than or equal to the number of even elements in $S\cap\{1,2,\dots,i\}$.
\end{enumerate}
}
\end{defi}

Notice that by Eq.\,(\ref{eqn:descent_top-bottom}), the set of descent tops and descent bottoms of a primary even-odd-descent permutation in $\X_{2n}$ is a signature in $[2n]$.

\begin{defi} \label{def:signature-vector} {\rm
Let $S$ be a signature in $[2n]$. For each $i\in [2n]$, let $f(i)$ ($g(i)$, respectively) be the number of odd (even, respectively) elements in $S$ less then $i$. We associate $S$ with a vector $(s(1),s(2),\dots,s(2n))$ defined by\begin{equation} \label{eqn:vector-weight}
s(i)=\left\{
\begin{array}{ll}
f(i)-g(i)+1 &\mbox{if $i\in S$ is odd;} \\
f(i)-g(i)   &\mbox{if $i\in S$ is even;} \\
f(i)-g(i)+1 &\mbox{if $i\not\in S$.}
\end{array}
\right.
\end{equation}
}
\end{defi}

\smallskip
\begin{exa} \label{exa:signature-vector} {\rm
Take the signature $S=\{1,3,9\}\cup\{8,12,14\}$ in $\{1,2,\dots,14\}$. The associated vector is $(1,2,2,3,3,3,3,2,2,3,3,2,2,1)$, as shown in the diagram of Figure \ref{fig:admissible-diagram-even-odd}(a).
}
\end{exa}

\begin{figure}[ht]
\begin{center}
\psfrag{1}[][][0.85]{$1$}
\psfrag{2}[][][0.85]{$2$}
\psfrag{3}[][][0.85]{$3$}
\psfrag{4}[][][0.85]{$4$}
\psfrag{5}[][][0.85]{$5$}
\psfrag{6}[][][0.85]{$6$}
\psfrag{7}[][][0.85]{$7$}
\psfrag{8}[][][0.85]{$8$}
\psfrag{9}[][][0.85]{$9$}
\psfrag{10}[][][0.85]{$10$}
\psfrag{11}[][][0.85]{$11$}
\psfrag{12}[][][0.85]{$12$}
\psfrag{13}[][][0.85]{$13$}
\psfrag{14}[][][0.85]{$14$}
\psfrag{2.1t}[][][0.85]{$1\cdot 2t$}
\psfrag{3.2t}[][][0.85]{$2\cdot 3t$}
\psfrag{3.2}[][][0.85]{$3\cdot 2$}
\psfrag{2.1}[][][0.85]{$2\cdot 1$}
\psfrag{3^2}[][][0.85]{$3^2$}
\psfrag{i:}[][][0.85]{$i:$}
\psfrag{s(i):}[][][0.85]{$s(i):$}
\includegraphics[width=4.75in]{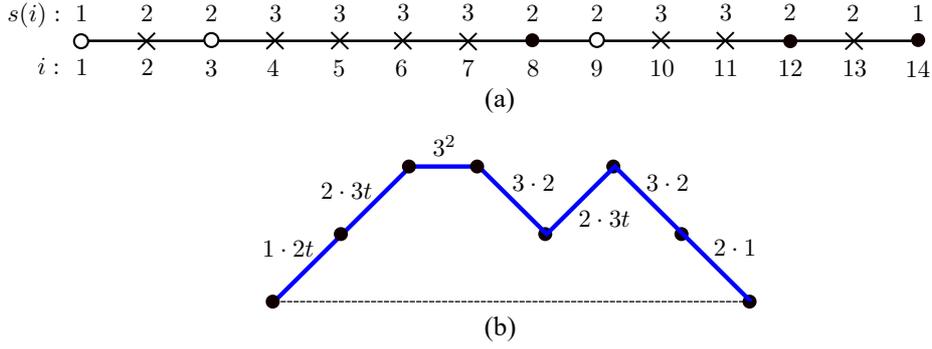}
\end{center}
\caption{\small The vector and weighted path associated with the signature in Example \ref{exa:signature-vector}.}
\label{fig:admissible-diagram-even-odd}
\end{figure}

\begin{pro} \label{pro:product-s(i)} For any signature $S\subset [2n]$, let $(s(1),\dots,s(2n))$ be the associated vector. Then the number of primary even-odd-descent permutations in $\X_{2n}$ with the signature $S$ as the set of descent tops and descent bottoms is given by 
\[
\prod_{i=1}^{2n} s(i).
\]
\end{pro}

Consider the following set of sequences determined by a signature $S$ in $[2n]$
\begin{equation} \label{eqn:B-code-for-S}
\{(b(1),\dots,b(2n))\, |\, 1\le b(i)\le s(i), 1\le i\le 2n\}.
\end{equation}
Let $\widehat{S}$ be the set of primary even-odd-descent permutations $\sigma$ in $\X_{2n}$ such that the set of descent tops and descent bottoms of $\sigma$ is $S$. 
To prove Proposition \ref{pro:product-s(i)}, we shall establish a bijection  $(b(1),\dots,b(2n))\mapsto \sigma$ of the set in Eq.\,(\ref{eqn:B-code-for-S}) onto $\widehat{S}$. 

Suppose the set $S$ consists of $k$ odd elements and $k$ even elements.  Given a sequence $(b(1),b(2),\dots,b(2n))$, we first construct the subword $\pi=x_1x_2\cdots x_{2k}$ of $\sigma$ consisting of the descent tops and descent bottoms, which is a down-up permutation, i.e., $x_1>x_2<\cdots<x_{2k-1}>x_{2k}$. Then we construct the corresponding permutation $\sigma$ by inserting the rest of elements into  $\pi$ as increasing runs.

\smallskip
\noindent
{\bf Algorithm B.}
\begin{enumerate}
\item[(B1)]  Let $y_1<y_2<\cdots<y_{2k}$ be the elements of $S$ in increasing order. 
We construct a sequence $\pi_1,\pi_2,\dots,\pi_{2k}=\pi$ of words, where $\pi_1=y_1$ and $\pi_i$ is obtained by inserting the element $y_i$ into $\pi_{i-1}$ for $2\le i\le 2k$. 
Note that $\pi_{i-1}$ contains $f(y_i)$ odd elements, and hence $f(y_i)+1$ spaces. By a \emph{space} of $\pi_{i-1}$ we mean the position to the left of the first odd entry, between two odd entries, or to the right of $\pi_{i-1}$. 
Among them, $g(y_i)$ spaces have been occupied by the even elements in $\pi_{i-1}$. 
There are $f(y_i)-g(y_i)+1$ unoccupied spaces in $\pi_{i-1}$, indexed by $1,2,\dots,f(y_i)-g(y_i)+1$ from left to right. 
We insert the element $y_i$ at the $b(y_i)$-th unoccupied space of $\pi_{i-1}$. 
Note that in the case where $y_i$ is even, we have $b(y_i)\leq f(y_i)-g(y_i)$, and that as a descent top, $y_i$ will not be at the position to the right of $\pi_{i-1}$.

\item[(B2)]  Assume $x_0=0$ and $x_{2k+1}=\infty$. For each element $y\in [2n]\setminus S$, an ascent $(x_{2j},x_{2j+1})$ of $\pi$ is \emph{feasible} relative to $y$ if $x_{2j}<y<x_{2j+1}$. Note that there are $f(y)-g(y)+1$ feasible ascents relative to $y$. We insert the element $y$ into the $b(y)$-th feasible ascent from left to right. Those elements inserted in the same ascent of $\pi$ are arranged in increasing order. 
\end{enumerate}
Notice that by Eqs.\,(\ref{eqn:vector-weight}) and (\ref{eqn:B-code-for-S}), the corresponding permutation $\sigma\in\widehat{S}$ is well defined.
Note that $\des(\sigma)=|S|/2$ for all $\sigma\in\widehat{S}$.

\smallskip
\begin{exa} \label{exa:signature-vector-even-odd-continued} {\rm
Using the signature $S$ in Example \ref{exa:signature-vector}, we construct the permutation $\sigma$ corresponding to the sequence $(b(1),\dots,b(14))$ $=(1, 1, 2, 3, 2, 2, 2, 1, 1, 2, 3, 2, 1, 1)$, with descent tops $\{8,12,14\}$ and descent bottoms $\{1,3,9\}$. The construction of the words $\pi_1,\dots,\pi_6$ is shown in Table \ref{tab:step-by-step-even-odd}, where the unoccupied spaces of $\pi_{i-1}$ are indicated by dots. Since $b(2)=1$, $b(4)=3$, $b(5)=b(6)=b(7)=2$, $b(10)=2$, $b(11)=3$, and $b(13)=1$, the requested permutation $\sigma$  is $\sigma=2\: 8\: 1\: 5\: 6\: 7\: 13\: 14\: 9\: 10\: 12\: 3\: 4\: 11$, as shown in Figure \ref{fig:encoding-diagram-even-odd}.  
}
\end{exa}

\begin{table}[ht]
\caption{The construction of the word $\pi$ in Example \ref{exa:signature-vector-even-odd-continued}.}
\centering
\begin{tabular}{cccrc|cr}
 \hline
  $y_i$ & $b(y_i)$ &  & \multicolumn{1}{c} {$\pi_{i-1}$} & & & \multicolumn{1}{c} {$\pi_i$}\\
\hline
  1     & 1        &  &     &  &  & 1  \\[0.5ex]
  3     & 2        &  &   . 1 .  & & & 1\: 3 \\[0.5ex]
  8     & 1        &  &   . 1 . 3  . &  &  & 8\: 1\: 3\\[0.5ex]
  9     & 1        &  &   8\, 1 .  3  . &  &  & 8\: 1\: 9\: 3  \\[0.5ex]
  12    & 2        &  &   8\, 1  . 9\, . 3  . &  &  & 8\: 1\:   9\: 12\: 3  \\[0.5ex]
  14    & 1        &  &  8\, 1  . 9\, 12\, 3  .  & & & 8\: 1\: 14\: 9\: 12\: 3   \\
\hline
\end{tabular}
\label{tab:step-by-step-even-odd}
\end{table}

\begin{figure}[ht]
\begin{center}
\psfrag{inf}[][][0.85]{$\infty$}
\includegraphics[width=3.3in]{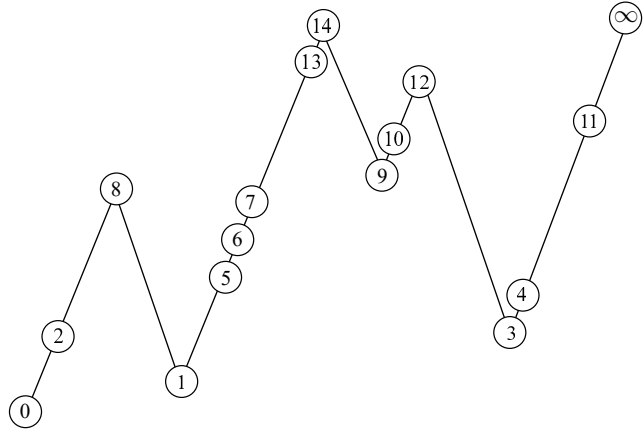}
\end{center}
\caption{\small The requested permutation $\sigma$ in Example \ref{exa:signature-vector-even-odd-continued}.}
\label{fig:encoding-diagram-even-odd}
\end{figure}

To construct the inverse map, given a primary even-odd-descent permutation $\sigma'$ in $\widehat{S}$, the sequence $(b'(1),\dots,b'(2n))$ corresponding to $\sigma'$ is given by
\begin{equation} \label{eqn:B'=c-d}
b'(i)=c(i)-d(i)+1,
\end{equation}
where $c(i)$ ($d(i)$, respectively) is the number of descent bottoms (descent tops, respectively) less than $i$ and on the left of the element $i$ in $\sigma'$. The proof of Proposition \ref{pro:product-s(i)} is completed.

\smallskip
\begin{exa} \label{exa:inverse-map} {\rm Following Example \ref{exa:signature-vector-even-odd-continued}, let $\sigma'=2\: 8\: 1\: 5\: 6\: 7\: 13\: 14\: 9\: 10\: 12\: 3\: 4\: 11\in\widehat{S}$, with descent bottoms $\{1,3,9\}$ and descent tops $\{8,12,14\}$.  The values of $c(i)$ and $d(i)$ are shown below. The sequence corresponding to $\sigma'$ is $(b'(1),\dots,b'(14))=(1, 1, 2, 3, 2, 2, 2, 1, 1, 2, 3, 2, 1, 1)$. 
\[
\begin{array}{c|cccccccccccccc}

 i & 1 &  2 & 3 & 4 & 5 & 6 & 7 &  8  & 9 & 10 & 11 & 12 & 13 & 14\\
\hline
 c(i)  & 0 & 0 & 1 & 2 & 1 & 1 & 1 & 0 & 1 & 2 & 3 & 2 & 1 & 1\\
 d(i)  & 0 & 0 & 0 & 0 & 0 & 0 & 0 & 0 & 1 & 1 & 1  & 1 & 1 & 1

\end{array}
\]
}
\end{exa}

\subsection{Continued fractions and weighted Motzkin paths}

A \emph{Motzkin path} of length $n$ is a lattice path from the origin to the point $(n,0)$ staying weakly above the $x$-axis, using the \emph{up step} $(1,1)$, \emph{down step} $(1,-1)$, and \emph{level step} $(1,0)$.  For a Motzkin path $M=z_1z_2\cdots z_n$ with a weight function $\rho$ on the steps, the \emph{weight} of $M$, denoted by $\rho(M)$, is defined to be the product of its step weights. 
The \emph{height} of a step $z_j$ is the $y$-coordinate of the starting point of $z_j$. 
Making use of Flajolet's formula \cite[Propositions 7A, 7B]{Flaj}, the generating function for the weighted count of the Motzkin paths can be expressed as a continued fraction.

\begin{thm} {\rm\bf (Flajolet)} \label{thm:Flajolet} For $h\ge 0$, let $a_h$, $b_h$ and $c_h$ be polynomials such that each monomial has coefficient 1. Let $\M_n$ be the set of weighted Motzkin paths of length $n$ such that the weight of an up step  (down step or level step, respectively) at height $h$ is one of the monomials appearing in $a_h$ ($b_h$ or $c_h$, respectively). Then the following continued fraction expansion holds:
\begin{equation} \label{eqn:cf-Motzkin}
\sum_{n\ge 0}\left(\sum_{M\in\M_n} \rho(M) \right)x^n
      =\frac{1}{1-c_0x}{{}\atop{-}}\frac{a_0b_1x^2}{1-c_1x}{{}\atop{-}}\frac{a_1b_2x^2}{1-c_2x}{{}\atop{-}}\cdots.
\end{equation}
\end{thm}

\smallskip
Given a signature $S\subset [2n]$, let $(s(1),\dots,s(2n))$ be the vector associated with $S$. With each odd element in $S$ assigned a variable $t$, we define the \emph{weight} of $S$ by
\begin{equation} \label{eqn:weight-of-S}
\left(\prod_{i=1}^{2n} s(i) \right)t^{|S|/2}.
\end{equation}
Let $\SSS_{2n}$ denote the set of weighted signatures in $[2n]$. We shall enumerate the signatures in terms of weighted Motzkin paths. Let $\U$, $\DD$ and $\L$ denote an up step, a down step and a level step in a Motzkin path, accordingly, and let $z^{(h)}$ denote a step $z$ at height $h$ for $z\in\{\U,\DD,\L\}$. Let $\M_n$ denote the set of Motzkin paths $M$ of length $n$ with a weight function $\rho$ on the steps of $M$ given by
\begin{equation} \label{eqn:weight-Motzkin-path-S}
\begin{aligned}
\rho(\U^{(h)}) &=(h+1)(h+2) t, &\mbox{for $h\ge 0$;}\\
\rho(\L^{(h)}) &=(h+1)^2,      &\mbox{for $h\ge 0$;} \\
\rho(\DD^{(h)})&=(h+1)h,       &\mbox{for $h\ge 1$.}
\end{aligned}
\end{equation}

\begin{lem} \label{lem:weight-preserving} There is a weight-preserving bijection $S\mapsto M$ of $\SSS_{2n}$ onto $\M_n$.
\end{lem}

\begin{proof} Let $(s(1),\dots,s(2n))$ be the vector associated with $S$. We construct a Motzkin path $M=z_1z_2\cdots z_n$ from $S$ by 
\begin{equation}
z_j=\left\{
\begin{array}{ll}
\U &\mbox{if $2j-1\in S$;} \\[0.8ex]
\DD &\mbox{if $2j\in S$;} \\[0.8ex]
\L &\mbox{if $2j-1, 2j\not\in S$,}
\end{array}
\right.
\end{equation}
with a weight determined from $(s(1),\dots,s(2n))$ by
\begin{equation}
\rho(z_j)=\left\{
\begin{array}{ll}
s(2j-1)s(2j)t &\mbox{if $z_j=\U$;} \\[0.8ex]
s(2j-1)s(2j)  &\mbox{if $z_j\in\{\DD,\L\}$.} \end{array}
\right.
\end{equation}
Notice that if $z_j=\U^{(h)}$ then $2j-1\in S$ and $2j\not\in S$. By Eq.\,(\ref{eqn:vector-weight}), we have $s(2j-1)=f(2j-1)-g(2j-1)+1=h+1$ and $s(2j)=s(2j-1)+1=h+2$. Hence $\rho(z_j)=(h+1)(h+2)t$. If $z_j=\L^{(h)}$ then $2j-1,2j\not\in S$ and $s(2j)=s(2j-1)=f(2j-1)-g(2j-1)+1=h+1$. Hence $\rho(z_j)=(h+1)^2$. Moreover, if $z_j=\DD^{(h)}$ then $2j-1\not\in S$ and $2j\in S$. Note that $s(2j-1)=f(2j-1)-g(2j-1)+1=h+1$ and $s(2j)=f(2j)-g(2j)=h$. Hence $\rho(z_j)=(h+1)h$. By Eq.\,(\ref{eqn:weight-Motzkin-path-S}), we have $M\in\M_n$. Moreover, the weight of $M$,
\[
\rho(M)=\prod_{j=1}^{n} \rho(z_j)=\left(\prod_{i=1}^{2n} s(i)\right)t^{|S|/2},
\] is equal to the weight of $S$.

The inverse map $M\mapsto S$ can be constructed straightforward by the reverse operation. The assertion follows.
\end{proof}

\begin{exa} {\rm
Following Example \ref{exa:signature-vector}, the corresponding weighted Motzkin path of the signature $S=\{1,3,9\}\cup\{8,12,14\}$ is shown in Figure \ref{fig:admissible-diagram-even-odd}(b).
}
\end{exa} 

\smallskip
\noindent
\emph{Proof of Theorem \ref{thm:even-odd-descent-result}:}
Let $\PP_{2n}$ be the set of primary even-odd-descent permutations in $\X_{2n}$. By Eq.\,(\ref{eqn:representative}), we have
\begin{align*}
\sum_{\sigma\in\X_{2n}} t^{\des(\sigma)} &= \sum_{\pi\in\PP_{2n}} \left( \sum_{\sigma\in\Hop(\pi)} t^{\des(\pi)}\right)\\
&=\sum_{\pi\in\PP_{2n}} t^{\des(\pi)}(1+t)^{n-2\,\des(\pi)}.
\end{align*}
By Proposition \ref{pro:product-s(i)} and Lemma \ref{lem:weight-preserving}, we have
\begin{align*}
\sum_{\pi\in\PP_{2n}} t^{\des(\pi)} &= \sum_{j=0}^{\lfloor n/2\rfloor} \gamma_{n,j} t^j \\
&= \sum_{S\in\SSS_{2n}} \left(\prod_{i=1}^{2n} s(i) \right)t^{|S|/2} \\
&= \sum_{M\in\M_n} \rho(M).
\end{align*}
By Flajolet's theory of continued fractions \cite[Proposition 7B]{Flaj}, we prove Eq.\,(\ref{eqn:CF-expansion}). This completes the proof of Theorem \ref{thm:even-odd-descent-result}.
\qed 


%
\subsection{Permutations with only even-odd drops}
The descent number is closely related to the permutation statistic of drop.  For any $\sigma=\sigma_1\cdots\sigma_n\in \mathfrak{S}_n$, a \emph{drop} in $\sigma$ is an ordered pair $(i,\sigma_i)$ such that $i>\sigma_i$. The element $i$ ($\sigma_i$, respectively) is called a \emph{drop top} (\emph{drop bottom}, respectively). The drop $(i,\sigma_i)$ is called \emph{even-odd}, \emph{odd-even}, \emph{odd-odd}, and \emph{even-even} if the parities of $(i, \sigma_i)$ is (even, odd), (odd, even), (odd, odd), and (even, even), accordingly.   Let $\R_{n}$ be the set of permutations in $\mathfrak{S}_{n}$ that contain only even-odd drops . It is known that $|\R_{2n}|$ is the $n$th median Genocchi number \cite[Corollary 6.2]{LW}.  Let $\drop(\sigma)$ denote the number of drops of $\sigma$. 

\begin{pro} \label{pro:descent-to-drop-bijection-even-odd} There is a bijection $\sigma\mapsto\sigma'$ of $\X_{2n}$ onto $\R_{2n}$ with $\drop(\sigma')=\des(\sigma)$.
\end{pro}

\begin{proof} To describe the map $\sigma\mapsto\sigma'$, we write $\sigma$ in its disjoint cycle notation. Within each cycle, order the entries so that the smallest entry appears last. Then order the cycles in increasing order of their minimal elements. Then upon removing the parentheses, $\sigma'$ is the resulting permutation, written in one-line notation.

To describe the inverse map, read $\sigma'$ from right to left, and insert a divider at the immediate right of each right-to-left minimum. Insert parentheses so that the entries between dividers form cycles. Then $\sigma$ is the resulting permutation, written in cycle notation.
\end{proof}

\smallskip
As a byproduct of Theorem \ref{thm:even-odd-descent-result}, we obtain the $\gamma$-positivity for the drop polynomial for $\R_{2n}$. 

\begin{defi} \label{def:primary-even-odd-drop} {\rm 
For any $\sigma\in\R_{2n}$ with drop tops $\{t_1,\dots,t_k\}$ and drop bottoms $\{b_1,\dots,b_k\}$ for some $k\ge 0$, we say that $\sigma$ is a \emph{primary even-odd-drop permutation} if for any $i,j$, 
\[
t_i>b_j \Rightarrow t_i-b_j\ge 3.
\]
}
\end{defi}

\begin{cor} \label{cor:even-odd-drop-corollary} For all $n\ge 1$, we have
\begin{equation} \label{eqn:R_n(t)}
\sum_{\sigma\in\R_{2n}} t^{\drop(\sigma)}=\sum_{j=0}^{\lfloor n/2\rfloor} \gamma_{n,j} t^j(1+t)^{n-2j},
\end{equation}
where $\gamma_{n,j}$ is the number of primary even-odd-drop permutations in $\R_{2n}$ with $j$ drops.
\end{cor}

\section{On permutations with only odd-odd descents} 
In this section we shall prove Theorem \ref{thm:odd-odd-descent-result}, using an analogous encoding approach as in the proof of Theorem \ref{thm:even-odd-descent-result}. The encoding schemes provide an algorithmic bijection between these two families of permutations; see Theorem \ref{thm:connection-X(2n)-W*(2n+1)}.

\subsection{Peak-hopping operation}
Let $\sigma=\sigma_1\cdots\sigma_{2n+1}\in\W^*_{2n+1}$. Recall that $\sigma$ contains only odd-odd descents and the last entry $\sigma_{2n+1}$ is odd. Notice that the peaks, valleys, and double descents of $\sigma$ are odd elements necessarily. With the convention $\sigma_0=\sigma_{2n+2}=0$, the element $\sigma_1$ ($\sigma_{2n+1}$, respectively) is a double ascent (double descent, respectively) if $\sigma_1<\sigma_2$ ($\sigma_{2n}>\sigma_{2n+1}$, respectively). An odd element of $\sigma$ is \emph{free} if it is either a double ascent or a double descent.
We partition the permutations in $\W^*_{2n+1}$ into equivalence classes by a \emph{peak-hopping} operation on free odd elements.
Note that the peak-hopping operation is similar to the Br\"and\'en's modified Foata-Strehl action~\cite{Branden_08} but just acts on free odd elements.

Given $\omega=\omega_1\cdots\omega_{2n+1}\in\W^*_{2n+1}$ with a free element $2r-1$, we shall construct a permutation $\omega'\in\W^*_{2n+1}$ such that $\art(\omega')=\art(\omega)$ and $|\des(\omega')-\des(\omega)|=1$ by the following process.

\smallskip
\noindent
{\bf Algorithm C.}

If $\omega_i=2r-1$ is a double ascent, i.e., $\omega_{i-1}<\omega_i<\omega_{i+1}$, then find the smallest $k>i$ such that $\omega_k>2r-1>\omega_{k+1}$, and set
\[
\omega'=\omega_1\cdots\omega_{i-1}\omega_{i+1}\cdots\omega_k\,(2r-1)\,\omega_{k+1}\cdots\omega_{2n+1}.
\]
Otherwise $\omega_i=2r-1$ is a double descent, i.e., $\omega_{i-1}>\omega_i>\omega_{i+1}$, then find the largest $k<i$ such that $\omega_k<2r-1<\omega_{k+1}$, and set
\[
\omega'=\omega_1\cdots\omega_{k-1}\,(2r-1)\,\omega_{k+1}\cdots\omega_{i-1}\omega_{i+1}\cdots\omega_{2n+1}.
\]

\smallskip
Let $\Hop(\omega)$ be the hop-equivalence class of $\omega$. Notice that if $\omega$ has $j$ valleys then it has $j+1$ peaks, and hence $n-2j$ free odd elements.  
Putting each free odd element of $\omega$ in an increasing run yields the unique primary odd-odd-descent permutation, say $\pi$, in $\Hop(\omega)$. We have
\begin{equation} \label{eqn:art-representative}
\sum_{\sigma\in \Hop(\omega)} q^{\art(\sigma)}t^{\des(\sigma)}
=q^{\art(\pi)}t^{\des(\pi)}(1+t)^{n-2\,\des(\pi)}.
\end{equation}

\smallskip
\begin{exa} \label{exa:peak-hopping} {\rm Let $\omega=2\: 8\: 12\: 14\: 15\: 5\: 3\: 4\: 6\: 7\: 11\: 9\: 10\: 13\: 1\in\W^*_{15}$. Note that the free elements of $\omega$ are $\{1,5,7\}$. The peak-hopping operation is illustrated in Figure \ref{fig:odd-hopping}. The unique primary odd-odd-descent permutation $\pi$ in $\Hop(\omega)$ is $\pi=1\: 2\: 5\: 8\: 12\: 14\: 15\: 3\: 4\: 6\: 7\: 11\: 9\: 10\: 13$.
}
\end{exa}

\begin{figure}[ht]
\begin{center}
\psfrag{1+t}[][][0.85]{$1+t$}
\includegraphics[width=3.2in]{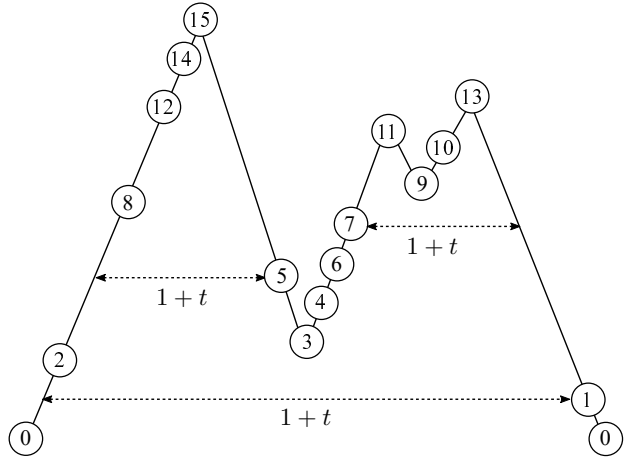}
\end{center}
\caption{\small An illustration for Example \ref{exa:peak-hopping}.}
\label{fig:odd-hopping}
\end{figure}


\subsection{Encoding primary odd-odd-descent permutations}
We shall enumerate the primary odd-odd-descent permutations by a classification of their peaks and valleys (Proposition \ref{pro:product-a(i)}). 

\begin{defi} {\rm
A set $A$ of odd elements, with each element colored in black or white, is called an \emph{admissible set} in $\{1,3,\dots,2n+1\}$ if it satisfies the following conditions.
\begin{enumerate}
\item The set $A$ consists of $k$ white elements (i.e., valleys) and $k+1$ black elements (i.e., peaks) for some $k\ge 0$. One of the black element is $2n+1$.
\item For each $i\in [2n]$, the number of white elements is greater than or equal to the number of black elements in $A\cap\{1,2,\dots,i\}$.
\end{enumerate}
Moreover, we associate an admissible set $A$ with a vector $(a(1),a(2),\dots,a(2n))$ defined by
\begin{equation} \label{eqn:vector-a(i)}
a(i) = f(i)-g(i)+1,
\end{equation}
where $f(i)$ ($g(i)$, respectively) is the number of white (black, respectively) elements  less then $i$ in $A$ for $1\le i\le 2n$.
}
\end{defi}

\begin{exa} \label{exa:admissible-vector-odd} {\rm
Let $A$ be the admissible set in $\{1,3,\dots,15\}$ with white elements $\{1,3,9\}$ and black elements $\{7,11,13,15\}$. The vector associated with $A$ is $(1,2,2,3,3,3,3,2,2,3,3,2,2,1)$, as shown in the diagram of Figure \ref{fig:admissible-diagram-odd-odd}(a).
}
\end{exa}

\begin{figure}[ht]
\begin{center}
\psfrag{1}[][][0.85]{$1$}
\psfrag{2}[][][0.85]{$2$}
\psfrag{3}[][][0.85]{$3$}
\psfrag{4}[][][0.85]{$4$}
\psfrag{5}[][][0.85]{$5$}
\psfrag{6}[][][0.85]{$6$}
\psfrag{7}[][][0.85]{$7$}
\psfrag{8}[][][0.85]{$8$}
\psfrag{9}[][][0.85]{$9$}
\psfrag{10}[][][0.85]{$10$}
\psfrag{11}[][][0.85]{$11$}
\psfrag{12}[][][0.85]{$12$}
\psfrag{13}[][][0.85]{$13$}
\psfrag{14}[][][0.85]{$14$}
\psfrag{2.1t}[][][0.85]{$[1]_q[2]_qt$}
\psfrag{3.2t}[][][0.85]{$[2]_q[3]_qt$}
\psfrag{a(i):}[][][0.85]{$a(i):$}
\psfrag{i:}[][][0.85]{$i:$}
\psfrag{3^2}[][][0.85]{$[3]_q^2$}
\psfrag{2.1}[][][0.85]{$[2]_q[1]_q$}
\psfrag{3.2}[][][0.85]{$[3]_q[2]_q$}
\includegraphics[width=4.75in]{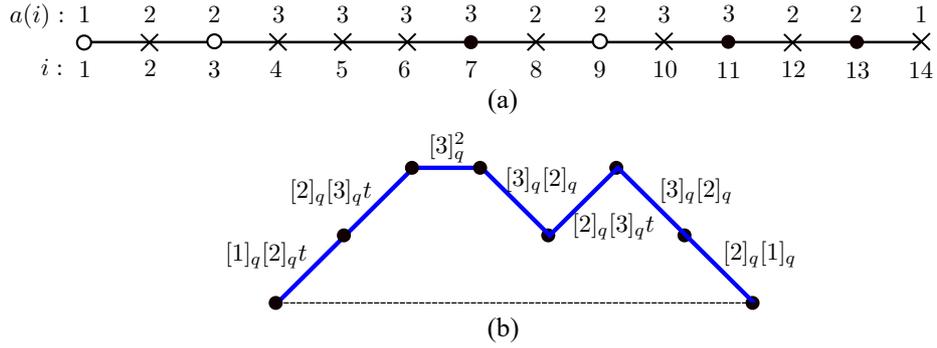}
\end{center}
\caption{\small The vector and weighted path associated with the admissible set in Example \ref{exa:admissible-vector-odd}.}
\label{fig:admissible-diagram-odd-odd}
\end{figure}

\begin{pro} \label{pro:product-a(i)} For any admissible set $A\subset \{1,3,\dots,2n+1\}$, let $(a(1),\dots,a(2n))$ be the vector associated with $A$, and let $\widehat{A}$ be the set of primary odd-odd-descent permutations $\sigma$ in $\W^*_{2n+1}$ such that the set of peaks and valleys of $\sigma$ is $A$. We have
\[
\sum_{\sigma\in\widehat{A}} q^{\art(\sigma)}=\prod_{i=1}^{2n} [a(i)]_q.
\]
\end{pro}

Consider the following set of sequences determined by an admissible set $A$ in $\{1,3,\dots,2n+1\}$
\begin{equation} \label{eqn:B-code-for-A}
\{(b(1),\dots,b(2n))\, |\, 0\le b(i)\le a(i)-1, 1\le i\le 2n\}.
\end{equation}
To prove Proposition \ref{pro:product-a(i)}, it suffices to establish a bijection  $(b(1),\dots,b(2n))\mapsto \sigma$ of the set in Eq.\,(\ref{eqn:B-code-for-A}) onto $\widehat{A}$ with $\art(\sigma)=b(1)+\cdots+b(2n)$. 

Suppose the set $A$ consists of $k$ white elements and $k+1$ black elements.  Given a sequence $(b(1),b(2),\dots,b(2n))$, we first construct the subword $\pi=x_1x_2\cdots x_{2k+1}$ of $\sigma$ consisting of peaks and valleys.
Since $\sigma$ contains no double descent, the word $\pi$ forms a down-up permutation, i.e., $x_1>x_2<x_3>\cdots>x_{2k}<x_{2k+1}$. Then we construct the corresponding permutation $\sigma$ by inserting the rest of elements into $\pi$ as increasing runs.
Therefore, $\des(\sigma)=k=(|A|-1)/2$.

\smallskip
\noindent
{\bf Algorithm D.}
\begin{enumerate}
\item[(D1)]  Let $y_1<y_2<\cdots<y_{2k+1}=2n+1$ be the elements of $A$ in increasing order. We construct a sequence $\pi_1,\pi_2,\dots,\pi_{2k+1}=\pi$ of words, where $\pi_1=y_1$ and $\pi_i$ is obtained by inserting the element $y_i$ into $\pi_{i-1}$ for $2\le i\le 2k+1$. By a \emph{space} of $\pi_{i-1}$ we mean the position to the left of the first white entry of $\pi_{i-1}$, between two white entries of $\pi_{i-1}$, or to the right of $\pi_{i-1}$. Since $\pi_{i-1}$ consists of $f(y_i)$ white elements and $g(y_i)$ black elements, there are $f(y_i)-g(y_i)+1$ unoccupied spaces, assigned a weight of $0,1,2,\dots,f(p_i)-g(p_i)$ from left to right. For $2\le i\le 2k$, we insert the element $y_i$ at the space with weight $b(y_i)$. The element $2n+1$ is then inserted at the only available space of $\pi_{2k}$, which is of zero weight. 

\item[(D2)]  Assume $x_0=0$. For each element $y\in [2n]\setminus A$, an ascent $(x_{2j},x_{2j+1})$ of $\pi$ is \emph{feasible} relative to $y$ if $x_{2j}<y<x_{2j+1}$. Note that there are $f(y)-g(y)+1$ feasible ascents relative to $y$, assigned a weight of $0,1,2,\dots,f(y)-g(y)$ from left to right. We insert the element $y_i$ into the ascent with weight $b(y_i)$. Those elements inserted in the same ascent of $\pi$ are arranged in increasing order.
\end{enumerate}
Notice that for each $i\in [2n]$, the weight $b(i)$ coincides with $v(i)-p(i)$, where $v(i)$ ($p(i)$, respectively) is the number of white (black, respectively) elements less than $i$ and on the left of $i$. Hence $\art(\sigma)=b(1)+\cdots+b(2n)$.

\smallskip
\begin{exa} \label{exa:admissible-vector-odd-odd-continued} {\rm
Using the admissible set in Example \ref{exa:admissible-vector-odd}, we construct the permutation $\sigma$ corresponding to the sequence $(b(1),\dots,b(14))$ $=(0, 0, 1, 2, 1, 1, 1, 0, 0, 1, 2, 1, 0, 0)$, with peaks $\{7,11,13,15\}$ and valleys $\{1,3,9\}$. The construction of the words $\pi_1,\dots,\pi_7$ is shown in Table \ref{tab:step-by-step-odd}, where the peaks are indicated in bold face. Since $b(2)=0$, $b(4)=2$, $b(5)=b(6)=1$, $b(8)=0$, $b(10)=1$, $b(12)=1$, and $b(14)=0$, the requested permutation $\sigma$ is $\sigma=2\: 8\: 13\: 9\: 10\: 12\:  14\: 15\: 1\: 5\: 6\: 7\: 3\: 4\: 11$, as shown in Figure \ref{fig:encoding-diagram-odd-odd}.  
}
\end{exa}

\begin{table}[ht]
\caption{The construction of the word  $\pi$ in Example \ref{exa:admissible-vector-odd-odd-continued}}
\centering
\begin{tabular}{ccc|crccc}
 \hline
  $y_i$ & $b(y_i)$ &  &  & \multicolumn{1}{c} {$\pi_i$} &  &   $v(y_i)$  & $p(y_i)$\\
\hline
  1     & 0        &  &  &    1    &    &  0  &  0\\[0.5ex]
  3     & 1        &  &  &  1\:  3   &  &  1  &  0\\[0.5ex]
  7     & 1        &  &  &  1\: \textbf{7}\: 3   &  &  1  &  0\\[0.5ex]
  9     & 0        &  &  &  9\:  1\: \textbf{7}\: 3  &  & 0  & 0\\[0.5ex]
  11    & 2        &  &  &  9\:  1\: \textbf{7}\: 3\: \textbf{11} &  &  3  & 1 \\[0.5ex]
  13    & 0        &  &  & \textbf{13}\: 9\:  1\: \textbf{7}\: 3\: \textbf{11} &  &  0  &  0\\[0.5ex]
  15    &          &  &  & \textbf{13}\: 9\: \textbf{15}\: 1\: \textbf{7}\: 3\: \textbf{11}  &  &  &\\
\hline
\end{tabular}
\label{tab:step-by-step-odd}
\end{table}

\begin{figure}[ht]
\begin{center}
\includegraphics[width=3.2in]{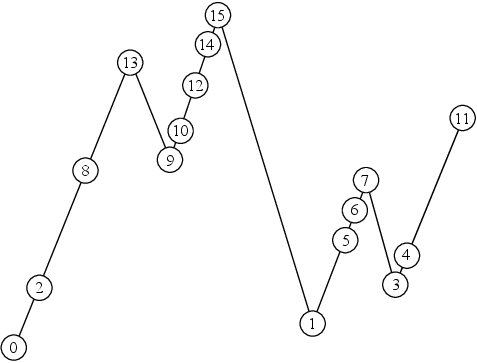}
\end{center}
\caption{\small The requested permutation in Example \ref{exa:admissible-vector-odd-odd-continued}.}
\label{fig:encoding-diagram-odd-odd}
\end{figure}

To construct the inverse map, given a primary odd-odd-descent permutations $\sigma'\in\widehat{A}$, the sequence $(b'(1),\dots,b'(2n))$ corresponding to $\sigma'$ is given by
\begin{equation} \label{eqn:B'=v-p}
b'(i)=v(i)-p(i),
\end{equation}
where $v(i)$ ($p(i)$, respectively) is the number of valleys (peaks, respectively) less than $i$ and on the left of the element $i$ in $\sigma'$.
The proof of Proposition \ref{pro:product-a(i)} is completed.

\subsection{Enumeration of primary odd-odd-descent permutations} 
For any admissible set $A$ in $\{1,3,\dots,2n+1\}$ with vector $(a(1),\dots,a(2n))$, we define the \emph{weight} of $A$ by
\begin{equation} \label{eqn:weight-of-A}
\left(\prod_{i=1}^{2n} [a(i)]_q \right)t^{(|A|-1)/2}.
\end{equation}
Let $\A_{2n+1}$ denote the set of weighted admissible sets in $\{1,3,\dots,2n+1\}$, and let $\M_n$ denote the set of Motzkin paths $M$ with a weight function $\rho$ on the steps of $M$ given by
\begin{equation} \label{eqn:weight-Motzkin-path-q-analog}
\begin{aligned}
\rho(\U^{(h)}) &=[h+1]_q[h+2]_q t, &\mbox{for $h\ge 0$;}\\
\rho(\L^{(h)}) &=[h+1]_q^2,      &\mbox{for $h\ge 0$;} \\
\rho(\DD^{(h)})&=[h+1]_q[h]_q,   &\mbox{for $h\ge 1$.}
\end{aligned}
\end{equation}

\begin{lem} \label{lem:weight-preserving-odd}There is a weight-preserving bijection $A\mapsto M$ of $\A_{2n+1}$ onto $\M_n$.
\end{lem}

\begin{proof} Let $(a(1),\dots,a(2n))$ be the vector associated with $A$. The corresponding Motzkin path $M=z_1z_2\cdots z_n$ is constructed from $A$ by 
\begin{equation}
z_j=\left\{
\begin{array}{ll}
\U &\mbox{if $2j-1\in A$ is a white element;} \\[0.8ex]
\DD &\mbox{if $2j-1\in A$ is a black element;} \\[0.8ex]
\L &\mbox{if $2j-1\not\in A$,}
\end{array}
\right.
\end{equation}
with a weight determined from $(a(1),\dots,a(2n))$ by
\begin{equation}
\rho(z_j)=\left\{
\begin{array}{ll}
[a(2j-1)]_q[a(2j)]_qt &\mbox{if $z_j=\U$;} \\[0.8ex]
[a(2j-1)]_q[a(2j)]_q  &\mbox{if $z_j\in\{\DD,\L\}$.} \end{array}
\right.
\end{equation}
By Eq.\,(\ref{eqn:vector-a(i)}), if the height of $z_j$ is $h$ then $a(2j-1)=f(2j-1)-g(2j-1)+1=h+1$. Notice that if $z_j=\U^{(h)}$ then $2j-1$ is a white element in $A$, and hence $a(2j)=a(2j-1)+1=h+2$. Thus  $\rho(z_j)=[h+1]_q[h+2]_qt$. If $z_j=\L^{(h)}$ then $2j-1\not\in A$ and hence $a(2j)=a(2j-1)=h+1$. Thus $\rho(z_j)=[h+1]_q^2$. If $z_j=\DD^{(h)}$ then $2j-1$ is a black element in $A$, and hence  $a(2j)=a(2j-1)-1=h$. Thus $\rho(z_j)=[h+1]_q[h]_q$. The weight of $M$,
\[
\prod_{j=1}^{n} \rho(z_j)=\left(\prod_{i=1}^{2n} [a(i)]_q\right)t^{(|A|-1)/2},
\]
is equal to the weight of $A$.

The inverse map $M\mapsto A$ can be constructed straightforward by the reverse operation. The assertion follows.
\end{proof}
 
\begin{exa} {\rm
Following Example \ref{exa:admissible-vector-odd}, let $A$ be the admissible set in $\{1,3,\dots,15\}$ with  white elements $\{1,3,9\}$ and black elements $\{7,11,13,15\}$. The corresponding weighted Motzkin path of $A$ is shown in Figure \ref{fig:admissible-diagram-odd-odd}(b).
}
\end{exa} 

\smallskip
\noindent
\emph{Proof of Theorem \ref{thm:odd-odd-descent-result}:} Let $\PP^*_{2n+1}$ be the set of primary odd-odd-descent permutations in $\W^*_{2n+1}$. By Eq.\,(\ref{eqn:art-representative}), we have
\begin{align*}
\sum_{\sigma\in\W^*_{2n+1}} q^{\art(\sigma)}t^{\des(\sigma)} &= \sum_{\pi\in\PP^*_{2n+1}} \left( \sum_{\sigma\in\Hop(\pi)} q^{\art(\sigma)}t^{\des(\pi)}\right)\\
&=\sum_{\pi\in\PP^*_{2n+1}} q^{\art(\pi)}t^{\des(\pi)}(1+t)^{n-2\,\des(\pi)}.
\end{align*}
By Proposition \ref{pro:product-a(i)} and Lemma \ref{lem:weight-preserving-odd}, we have
\begin{align*}
\sum_{\pi\in\PP^*_{2n+1}} q^{\art(\pi)}t^{\des(\pi)} &= \sum_{j=0}^{\lfloor n/2\rfloor} \gamma_{n,j}(q)\, t^j \\
&= \sum_{A\in\A_{2n+1}} \left(\sum_{\sigma\in\widehat{A}} q^{\art(\sigma)} \right)t^{(|A|-1)/2} \\
&= \sum_{A\in\A_{2n+1}} \left(\prod_{i=1}^{2n} [a(i)]_q \right)t^{(|A|-1)/2} \\
&= \sum_{M\in\M_n} \rho(M).
\end{align*}
By Theorem \ref{thm:Flajolet} and the weight in Eq.\,(\ref{eqn:weight-Motzkin-path-q-analog}), we prove Eq.\,(\ref{eqn:q-CF-expansion}). This completes the proof of Theorem \ref{thm:odd-odd-descent-result}.
\qed 

\smallskip
On the basis of the proofs of Theorems \ref{thm:even-odd-descent-result} and \ref{thm:odd-odd-descent-result}, we establish a bijection between $\W^*_{2n+1}$ and $\X_{2n}$.

\begin{thm} \label{thm:connection-X(2n)-W*(2n+1)}
There is a bijection $\sigma\mapsto\sigma'$ of $\W^*_{2n+1}$ onto $\X_{2n}$ with $\des(\sigma')=\des(\sigma)$.
\end{thm}

\begin{proof} We describe the construction of the bijection in two parts. 
The first part is for the primary permutations.
Given a primary odd-odd-descent $\omega\in\W^*_{2n+1}$, let $A$ be the set of peaks and valleys of $\omega$. Remove the element $2n+1$ from $A$. By the formula in Eq.\,(\ref{eqn:B'=v-p}), we encode $\omega$ with a sequence $(b(1),\dots,b(2n))$ in Eq.\,(\ref{eqn:B-code-for-A}). Then create a signature $S\subset [2n]$ from $A$ by setting  $2j-1\in S$ ($2j\in S$, respectively) if and only if $2j-1$ is a valley (peak, respectively) in $A$ for each $j\in [n]$, and create a sequence $(b'(1),\dots,b'(2n))$ in Eq.\,(\ref{eqn:B-code-for-S}) by  $b'(i)=b(i)+1$ for each $i\in [2n]$. Note that  $S$ and $A$ share the same associated vector. The corresponding primary even-odd-descent permutation $\omega'$ is then constructed from $S$ and $(b'(1),\dots,b'(2n))$ by algorithm $B$, with $\des(\omega')=|S|/2=(|A|-1)/2=\des(\omega)$. The inverse map can be constructed straightforward by the reverse operation.

The second part of the construction is within each hop-equivalence class. Suppose $\des(\omega)=k$. There are $n-2k$ free odd elements (free pairs, respectively) in $\omega$ ($\omega'$, respectively). Then we establish a bijection  between $\Hop(\omega)$ and $\Hop(\omega')$ by setting the $j$th smallest free pair of $\omega'$ to be saturated by algorithm A whenever we set the $j$th smallest free odd element of $\omega$ to be a double descent by algorithm C.
\end{proof}

\subsection{Cyclic permutations with only odd-odd drops} Let $\C_{n}$ denote the set of cyclic permutations in $\mathfrak{S}_n$ that contain only odd-odd drops. As a byproduct of Theorem \ref{thm:odd-odd-descent-result}, we obtain the $\gamma$-positivity for the drop polynomial for $\C_{2n+3}$.

\begin{pro} \label{pro:descent-to-drop-bijection-odd-odd} There is a bijection $\sigma\mapsto\sigma'$ of $\W^*_{2n+1}$ onto $\C_{2n+3}$ with $\drop(\sigma')=\des(\sigma)+1$.
\end{pro}

\begin{proof} To construct the map $\sigma\mapsto\sigma'$, create a word $\omega$ from $\sigma$ by incrementing every entry by 2 and adjoining the prefix of $1$ and $2$.  Then $\sigma'$ is the resulting cyclic permutation, written in cycle notation, by enclosing $\omega$ with parentheses entirely.

For $\sigma'\mapsto\sigma$, since $\sigma'$ contains only odd-odd drops, we have $\sigma'(1)=2$ and $\sigma'^{-1}(1)$ is odd. The standard cycle notation of $\sigma'$, starting with the element $1$, has an odd last element. Remove parentheses and the elements $1$ and $2$. Then $\sigma$ is the resulting permutation, written in one-line notation, by decrementing each remaining entry by 2.
\end{proof}

\smallskip
\begin{defi} \label{def:primary-odd-odd-drop} {\rm For any  $\sigma\in\mathfrak{S}_n$, the element $i$ is called a \emph{double drop} of $\sigma$ if $\sigma^{-1}(i)>i>\sigma(i)$, $1\le i\le n$. 
Given $\sigma\in\C_{2n+3}$, we say that $\sigma$ is a \emph{primary odd-odd-drop permutation} if $\sigma$ contains no double drop.
}
\end{defi}

\begin{cor} \label{cor:odd-odd-drop-corollary} For all $n\ge 1$, we have
\begin{equation} \label{eqn:C_2n+3)}
\sum_{\sigma\in\C_{2n+3}} t^{\drop(\sigma)-1}=\sum_{j=0}^{\lfloor n/2\rfloor} \gamma_{n,j} t^j(1+t)^{n-2j},
\end{equation}
where $\gamma_{n,j}$ is the number of primary odd-odd-drop permutations in $\C_{2n+3}$ with $j+1$ drops.
\end{cor}

\section{Concluding remarks}
It is a classical result that the Eulerian polynomial for $\mathfrak{S}_n$ is $\gamma$-positive. In this paper we prove that the $\gamma$-positivity is inherited by the descent polynomials for the subsets of permutations with descents of prescribed parities.
The results in Theorems \ref{thm:even-odd-descent-result} and \ref{thm:odd-odd-descent-result} can be easily extended to the remaining parity cases. Given a permutation $\sigma\in\mathfrak{S}_n$, we create a permutation in $\mathfrak{S}_{n+1}$ by incrementing each entry of $\sigma$ by one and adjoining the element 1 to the left of $\sigma$. By this operation, we obtain an analogous result of Theorem \ref{thm:even-odd-descent-result} for the permutations in $\mathfrak{S}_{2n+1}$ with only odd-even descents, and an analogous result of Theorem \ref{thm:odd-odd-descent-result} for the permutations in $\mathfrak{S}_{2n+2}$ with only even-even descents and an even last entry. See Table \ref{tab:summary-descent} for a summary.

\begin{table}[ht]
\caption{The $\gamma$-positivity for the descent polynomials of four classes of permutations.}
\centering
\begin{tabular}{|m{2.0in}|m{2.2in}|m{0.85in}|}
 \hline
  \multicolumn{1}{|c|}{objects}  & \multicolumn{1}{c|}{the $\gamma$-vector} & \multicolumn{1}{c|}{references}\\
\hline
the set of $\sigma\in\mathfrak{S}_{2n}$  with only even-odd  descents   &  For any descent top $t$ and descent bottom $b$ of $\sigma$, $t-b\ge 3$ if $t>b$. & Theorem \ref{thm:even-odd-descent-result}\\[2.5ex]
\hline  
the set of $\sigma\in\mathfrak{S}_{2n+1}$  with only odd-even  descents  & For any descent top $t$ and descent bottom $b$ of $\sigma$, $t-b\ge 3$ if $t>b$. & analogous to Theorem \ref{thm:even-odd-descent-result} \\[2.5ex]
\hline
the set of $\sigma\in\mathfrak{S}_{2n+1}$ with only odd-odd descents and an odd last element &  $\sigma$ contains no double descent, and $\sigma_{2n}<\sigma_{2n+1}$. & Theorem \ref{thm:odd-odd-descent-result}\\[2.5ex]
\hline
the set of $\sigma\in\mathfrak{S}_{2n+2}$  with only even-even  descents and an even last element &  $\sigma$ contains no double descent, and $\sigma_{2n+1}<\sigma_{2n+2}$. & analogous to Theorem \ref{thm:odd-odd-descent-result} \\
\hline
\end{tabular}
\label{tab:summary-descent}
\end{table}

By similar proofs of Corollaries \ref{cor:even-odd-drop-corollary} and  \ref{cor:odd-odd-drop-corollary}, we obtain analogous results for drop polynomials.
Note that the element 1 is a fixed point in every permutation with only even-even drops. A permutation in $\mathfrak{S}_{n}$ is called \emph{pseudo cyclic} if it contains a cycle of length $n-1$ and a fixed point. We summarize the $\gamma$-positivity for the drop polynomials of four classes of permutations in Table \ref{tab:summary-drop}. 

\begin{table}[ht]
\caption{The $\gamma$-positivity for the drop polynomials of four classes of permutations.}
\centering
\begin{tabular}{|m{2.0in}|m{2.0in}|m{1in}|}
 \hline
  \multicolumn{1}{|c|}{objects}  & \multicolumn{1}{c|}{the $\gamma$-vector} & \multicolumn{1}{c|}{references}\\
\hline
the set of $\sigma\in\mathfrak{S}_{2n}$  with only even-odd  drops   &  For any drop top $t$ and drop bottom $b$ of $\sigma$, $t-b\ge 3$ if $t>b$. & Corollary \ref{cor:even-odd-drop-corollary}\\[2.5ex]
\hline  
the set of $\sigma\in\mathfrak{S}_{2n+1}$  with only odd-even  drops  & For any drop top $t$ and drop bottom $b$ of $\sigma$, $t-b\ge 3$ if $t>b$. & analogous to Corollary \ref{cor:even-odd-drop-corollary}\\[2.5ex]
\hline
the set of cyclic permutations $\sigma$ in $\mathfrak{S}_{2n+3}$ with only odd-odd drops &  $\sigma$ contains no double drop. & Corollary \ref{cor:odd-odd-drop-corollary}\\[2.5ex]
\hline
the set of pseudo cyclic permutations $\sigma$ in $\mathfrak{S}_{2n+4}$ with only even-even drops &  $\sigma$ contains no double drop. & analogous to Corollary \ref{cor:odd-odd-drop-corollary}\\
\hline
\end{tabular}
\label{tab:summary-drop}
\end{table}

Genocchi numbers and median Genocchi numbers are ubiquitous in Combinatorics.
There are a number of objects counted by median Genocchi numbers such as Dumont derangements in $\mathfrak{S}_{2n}$ \cite[Corollary 2.4]{DR}, strict alternating pistols \cite{DV},  homogenized Linial arrangements \cite{Hetyei} and a class of permutations in $\mathfrak{S}_{2n}$ called collapsed permutations \cite{AHT}. We are interested in the statistics of these objects whose generating functions are $\gamma$-positive  ($q$-$\gamma$-positive, respectively).

\section*{Acknowledgements}
The authors thank the referees for reading the manuscript carefully and providing helpful suggestions. 
The authors were supported in part by
Ministry of Science and Technology (MOST), Taiwan under grants 110-2115-M-003-011-MY3 (S.-P. Eu), 109-2115-M-153-004-MY2 (T.-S. Fu), 108-2115-M-017-004 (H.-H. Lai), and 110-2115-M-153-004-MY2 (Y.-H. Lo).

\end{document}